\documentclass{amsart}
\usepackage[left=2.5cm,    
right=2.5cm,   
top=2cm,       
bottom=2cm]{geometry} 
\usepackage{amssymb}
\usepackage{amsmath}
\usepackage{mathtools}
\usepackage{booktabs}
\usepackage{float}
\usepackage{verbatim}

\usepackage{bm} 
\usepackage[colorlinks,linkcolor=cyan,citecolor=cyan]{hyperref}
\usepackage[hang,flushmargin]{footmisc} 
\usepackage[square,comma,sort&compress,numbers]{natbib} 
\usepackage{mathrsfs} 
\usepackage[font=footnotesize,skip=0pt,textfont=rm,labelfont=rm]{caption,subcaption} 

\usepackage{amsthm}
\newtheorem{theorem}{Theorem}[section]
\newtheorem{proposition}{Proposition}[section]
\newtheorem{lemma}{Lemma}[section]

\newtheorem{corollary}{Corollary}[section]
\newtheorem{remark}{Remark}[section]

\begin{document}
	\title{The stability of Yang-Mills connections on $\delta$-pinched manifolds}
	\author{Xiaoli Han}
	\address{Xiaoli Han\\ Math department of Tsinghua university\\ Beijing\\ 100084\\ China\\} \email{hanxiaoli@mail.tsinghua.edu.cn}
	\author{Yang Wen}
	\address{Yang Wen\\ Beijing International Center for Mathematical Research\\ Peking University\\ Beijing\\ 100871\\ China\\} \email{2406397066@pku.edu.cn}
	\begin{abstract}
		In this article, we establish pinching conditions under which all weakly stable Yang-Mills connections on compact manifolds are flat. As a corollary, we provide a dimension-dependent constant $\delta(n)$ and prove that there exist no non-flat weakly stable Yang-Mills connections on $\delta(n)$-pinched compact simply-connected Riemannian manifolds.
	\end{abstract}
	
	\begin{abstract}
		In this paper we prove that there is a neighborhood in the  $C^2$ topology of the usual metric on the
		Euclidean sphere $S^n (n\geq 5)$ such that there is no nontrivial weakly stable Yang-Mills connections for any metric $\tilde{g}$ in this neighborhood. We also study the stability of Yang-Mills connections on the warped product manifolds.
	\end{abstract}
	
	\vspace{1ex}
	
	\vspace{1ex}
	\maketitle
	\section{Introduction}\ 
	
	Let $(M,g)$ be an n-dimensional compact Riemannian manifold and $E$ be a vector bundle of rank $r$ over $M$ with structure group $G$, where $G$ is a compact Lie group. The classical Yang-Mills functional is given by
	\begin{align*}
		YM(\nabla)=\int_M|R^\triangledown|^2dV,
	\end{align*}
	where $\nabla$ is a connection on $E$ and $R^\triangledown$ is its curvature. We denote $d^\triangledown$ to be the exterior differential operator on $E$ induced by $\nabla$ and $\delta^\triangledown$ is its adjoint operator. The Euler-Lagrange equation of $YM$ is
	\begin{align*}
		\delta^\triangledown R^\triangledown=0
	\end{align*}
	and the critical point of $YM$ is called Yang-Mills connection. 
	
	We are interested in the stability of a Yang-Mills connection. We call $\nabla$ is weakly stable if the second variation of $YM$ is non-negative at $\nabla$, i.e.
	\begin{align*}
		\frac{d^2}{dt^2}YM(\nabla^t)\mid_{t=0}\ge0
	\end{align*}
	for any curve of connections $\nabla^t$ such that $\nabla^0=\nabla$.
	
	The study of geometric stability originated in minimal surface theory with the seminal Lawson-Simons conjecture \cite{LS}, which established the nonexistence of stable minimal $k$-submanifolds ($2\leq k\leq n-2$) in $\frac{1}{4}$-pinched compact Riemannian manifolds. Advancing the conjecture, Howard \cite{Ho} proved that there exists $\delta(n,p)\in(\frac14,1)$ such that there is no $p$-dimensional stable minimal surface in $n$-dimensional $\delta(n,p)$-pinched compact Riemannian manifolds. As a special case, Hu-Wei \cite{HW} proved that there is no stable minimal surface in a compact simply connected hypersurface in $\mathbb{R}^{n+1}$ or $S^{n+1}$ with $\frac15$-pinched curvature. Considering the perturbation of metrics, Franz-Trinca \cite{FT} established non-existence of closed stable minimal $k$-submanifolds in generic conformal spheres $(S^n,\tilde g)$ for dimensions $2\le k\le n-\delta^{-1}$, where $\delta$ dependents on $\tilde g$.
	
	Motivated by the fundamental analogies between minimal surfaces, harmonic maps, and Yang-Mills fields, significant efforts have been made to adapt techniques from minimal surface theory to investigate harmonic maps and Yang-Mills configurations. The harmonic mapping conjecture posits that every stable harmonic map from an arbitrary Riemannian manifold into a compact simply connected Riemannian manifold with $\frac14$ -pinched curvature must be constant. Building on foundational rigidity theorems for harmonic maps, 	Leung \cite{Le} established the non-existence of non-constant stable harmonic maps from arbitrary Riemannian domains to the standard sphere $S^n$ for $n \geq 3$. This result was significantly extended by Okayasu \cite{Ok} through the introduction of a dimension-independent curvature constraint. Specifically, Okayasu proved that for any simply connected compact 
	Riemannian manifold with $\delta$-pinched sectional curvature satisfying $\delta \geq \frac{5}{6}$, all stable harmonic maps into such target manifolds must be constant. 
	
	A natural conjecture arises: whether there can exist no non-flat weakly stable Yang-Mills connections on compact simply connected Riemannian manifolds with $\frac14$-pinched curvature? Bourguignon-Lawson \cite{BL} proved that any weakly stable Yang-Mills connections on $S^n$ ($n\ge5$) must be constant. They also proved that if the structure group $G=SU(2), SU(3)$, or $U(2)$, then the weakly stable Yang-Mills connection on $S^4$ is self-dual or anti self-dual. Kobayashi-Ohnita-Takeuchi \cite{KOT} proved that the Yang-Mills connection on any compact irreducible symmetric space other than $S^n$ ($n\ge5$), $P^2(\mathbb{C}ay)$, $E_6/F_4$ and compact simple Lie groups is a weakly stable Yang-Mills connection. Utilizing Lawson's constructive method,  Ni \cite{Ni} proved that the Morse index of $\operatorname{YM}$ on $S^n (n\ge5)$ is larger than $m+1$, and the smallest eigenvalue of the Jacobi operator of $\operatorname{YM}$ is not greater than $4-n$.
	
	The principal aim of this investigation is to determine a curvature pinching threshold that precludes the existence of non-flat weakly stable Yang-Mills connections. Inspired by Howard's \cite{Ho} variational construction, we prove the following theorem.
	\begin{theorem}\label{main theorem}
		Let $M$ be a compact $n$-dimensional Riemannian manifold with $n\ge5$ and $\operatorname{inj}(M) \geq c_n\pi$ for some $c_n \in (0,1]$. There exists $\delta = \delta(n, \operatorname{inj}(M)) > 0$ such that if the sectional curvature satisfies $\delta < K \leq 1$, then every weakly stable Yang-Mills connection on $M$ is flat.
	\end{theorem}
	If $(M, g)$ be a simply connected compact Riemannian manifold with dimensional $n>4$ and sectional curvature $K \in (\delta K_0, K_0]$ for constants $K_0, \delta > 0$. After the homothety $\tilde{g} = K_0 g$, we may assume the sectional curvature satisfies $K \in (\delta, 1]$. By Klingenberg injectivity radius estimate (see \cite{KW} 2.6.A.1), we have $c_n=1$ and thus $\operatorname{inj}(M)=\pi$. Therefore, we have the following deduction.
	\begin{corollary}\label{main corollary}
		Let $M$ be a simply connected compact Riemannian manifold with dimensional $n\ge5$. Then there exists $\delta(n)\in(0,1)$ such that if $M$ is $\delta(n)$-pinched, then any weakly stable Yang-Mills connection on $M$ is flat.
	\end{corollary}
	\begin{remark}
		Ohnita-Pan \cite{OP}  established results analogous to Corollary \ref{main corollary}, but their approach differs from ours, consequently yielding distinct constants $\delta(n)$.
	\end{remark}
	\section{Preliminary}
	\subsection{The connections and curvatures on vector bundles}\ 
	
	Assume $(M,g)$ be a $n$ dimensional compact Riemannian manifold and $\mathscr{X}(M)$ be the collection of vector fields on $M$. Let $E\to M$ be a rank $r$ vector bundle over $M$ with a compact Lie group $G$ as its structure group. We also assume $\langle\ ,\ \rangle$ be the Riemannian metric on $E$ compatible with the action of $G$ and $\mathfrak{g}_E$ be the adjoint bundle of $E$. Let $\nabla:\Omega^0(\mathfrak{g}_E)\to\Omega^1(\mathfrak{g}_E)$ be the connection on $E$ compatible with the metric $\langle\ ,\ \rangle$. Locally, $\nabla$ takes the form
	\begin{align*}
		\nabla=d+A,
	\end{align*}
	where $A\in\Omega^1(\mathfrak{g}_E)$.
	
	For any connection $\nabla$ on $E$, the corresponding curvature $R^\nabla$ is given by
	\begin{align*}
		R^\nabla=dA+\frac12[A\wedge A],
	\end{align*}
	where
	\begin{align*}
		\frac12[A\wedge A](X,Y)=[A(X),A(Y)].
	\end{align*}
	
	The induced inner product on $\Omega^p(\mathfrak{g}_E)$ is given by
	\begin{align*}
		\langle\phi,\psi\rangle=\frac12Tr(\phi^T\psi)
	\end{align*}
	for any $\phi,\psi\in\Omega^0(\mathfrak{g}_E)$ and
	\begin{align*}
		\langle\phi,\psi\rangle_g=\frac1{p!}\sum_{1\le i_1,...,i_p\le p}\langle\phi(e_{i_1},...,e_{i_p}),\psi(e_{i_1},...,e_{i_p})\rangle,
	\end{align*}
	for any $\phi,\psi\in\Omega^p(\mathfrak{g}_E)$, where $\{e_i\mid1\le i\le n\}$ is an orthogonal basis of $TS^n$ respect to the metric $g$. After integrating, we get the global inner product of $\Omega^p(\mathfrak{g}_E)$, that is
	\begin{align*}
		(\phi,\psi)_g=\int_{S^n}\langle\phi,\psi\rangle_g dV_g.
	\end{align*}
	The connection $\nabla$ induces a connection $d^\nabla:\Omega^p(\mathfrak{g}_E)\to\Omega^{p+1}(\mathfrak{g}_E)$ on $\Omega^p(\mathfrak{g}_E)$, and we assume $\delta^\nabla:\Omega^p(\mathfrak{g}_E)\to\Omega^{p-1}(\mathfrak{g}_E)$ be the formal adjoint operator of $d^\nabla$. On local coordinates, we have
	\begin{align*}
		d^\nabla\phi(X_1,...,X_{p+1})&=\sum_{i=1}^{p+1}(-1)^{i+1}\nabla_{X_i}\phi(X_1,...,\hat{X_i},...,X_{p+1}),\\
		\delta^\nabla\phi(X_1,...,X_{p-1})&=-\sum_{i=1}^n\nabla_{e_i}\phi(e_i,X_1,...,X_{p-1})
	\end{align*}
	for any $\phi\in\Omega^p(\mathfrak{g}_E)$. 
	
	We can define the Laplace–Beltrami operator $\Delta^\nabla$ by
	\begin{align*}
		\Delta^\nabla\phi=d^\nabla\delta^\nabla\phi+\delta^\nabla d^\nabla\phi
	\end{align*}
	and the rough Laplacian operator $\nabla^\ast\nabla$ by
	\begin{align*}
		\nabla^\ast\nabla\phi=-\sum_{i=1}^n(\nabla_{e_i}\nabla_{e_i}\phi-\nabla_{D_{e_i}e_i}\phi).
	\end{align*}
	For $\phi\in\Omega^1(\mathfrak{g}_E)$ and $\psi\in\Omega^2(\mathfrak{g}_E)$, define
	\begin{align*}
		\mathfrak{R}_g^\nabla(\phi)(X)&=\sum_{i=1}^n[R^\nabla(e_i,X),\phi(e_i)],\\
		\mathfrak{R}_g^\nabla(\psi)(X,Y)&=\sum_{i=1}^n[R^\nabla(e_i,X),\psi(e_i,Y)]-[R^\nabla(e_i,Y),\psi(e_i,X)].
	\end{align*}
	Then we have the following Bochner–Weizenböck formula first introduced by Bourguignon-Lawson.
	\begin{theorem}\cite{BL}
		For any $\phi\in\Omega^1(\mathfrak{g}_E)$ and $\psi\in\Omega^2(\mathfrak{g}_E)$, we have
		\begin{align*}
			&\Delta^\nabla\phi=\nabla^\ast\nabla+\phi\circ\operatorname{Ric}+\mathfrak{R}_g^\nabla(\phi),\\
			&\Delta^\nabla\psi=\nabla^\ast\nabla+\psi\circ(\operatorname{Ric}\wedge Id+2R_M)+\mathfrak{R}_g^\nabla(\psi),
		\end{align*}
		where $R_M$ is the curvature tensor of $M$ and
		\begin{itemize}
			\item $\operatorname{Ric} :TM\to TM$ is the Ricci transformation defined by
			\[\operatorname{Ric}\left( X \right) =\sum_jR_M(X,e_j)e_j ,\]
			\item $ \operatorname{Ric}\wedge \operatorname{Id} $ is the extension of the Ricci transformation $\operatorname{Ric}$ to $\wedge^2 TM$ given by
			\[ \operatorname{Ric}\wedge \operatorname{Id}(X,Y)=\operatorname{Ric}\wedge \operatorname{Id}(X\wedge Y)=\operatorname{Ric}(X)\wedge Y+X\wedge \operatorname{Ric}(Y),\]
			\item The composite map $\psi\circ R_M :\wedge^2 TM \to \Omega^0 \left( \mathfrak{g}_E \right)  $ is defined by
			\[\psi\circ R_M(X,Y)=\psi\circ R_M(X\wedge Y )  =\frac{1}{2}\sum_{j=1}^{n} \psi( e_j, R_M(X,Y)e_j) .\]
		\end{itemize}
	\end{theorem}
	\section{Estimate for the second variation of Yang-Mills functional}\label{sec:estimate of 2nd var}
	\subsection{Setting and Notation}\ 
	
		In this section, we will compute the second variation of the Yang-Mills functional along special variational directions. Let $M$ be a compact $n$-dimensional ($n\ge5$) Riemannian manifold with sectional curvature $K\in(\delta,1]$ for some constant $\delta\in(0,1)$ and $R:=\operatorname{inj}(M)\geq c_n\pi$ for some $c_n \in (0,1]$. Assume $\nabla$ be a Yang-Mills connection and $\nabla^t=\nabla+tB$ be a smooth curve in the affine space of connections for some $B\in\Omega^1(\mathfrak{g}_E)$. The second variation is
		\begin{equation}
			\mathscr{L}(B):=\frac{d^2}{d t^2} YM(\nabla+tB)|_{t=0}=\int_{S^n}\langle\delta^\triangledown d^\triangledown B+\mathfrak{R}^\triangledown(B), B\rangle dV.
		\end{equation}
		By direct calculation, we can prove the following lemma.
		\begin{lemma}\label{second variation}
			For any vector field $V\in\mathscr{X}(M)$, we have
			\begin{align*}
				\mathscr{L}(i_VR^\triangledown)=\int_M&\langle R^\triangledown(R(e_i,V)e_i,e_j)+R^\triangledown(e_i,R(e_i,V)e_j)-\nabla_{D_{e_i}V}R^\triangledown(e_i,e_j)-\nabla_{e_i}R^\triangledown(D_{e_i}V,e_j)+R^\triangledown(D^\ast DV,e_j)\\
				&+R^\triangledown(D^2_{e_i,e_j}V,e_i),R^\triangledown(V,e_j)\rangle dV.
			\end{align*}
		\end{lemma}
		Fix a point $y\in M$,  The geodesic distance function $\rho_y(x):=dist(x,y)$ is smooth when $x\notin Cut(y)$, the cut locus of $y$. We also define
		\begin{align*}
			f_R(t)=\left\{
			\begin{array}{ll}
				-\cos(\frac\pi Rt),&|t|\le R,\\
				1,&|t|\ge R
			\end{array}
			\right.
		\end{align*}
		be a $C^2$ function. Let $V^R_y$ be the vector field defined by
		\begin{align}\label{V_y}
			V^R_y(x)=\operatorname{grad}(f_R\circ\rho_y)={f_R}'(\rho_y)\operatorname{grad}(\rho_y).
		\end{align}
		Leveraging continuity in the second variation $\mathscr{L}(i_{V_y^R})$ with respect to $R$, we always assume $R=\pi$ in the following calculations. Let $f:=f_\pi$ and $V_y:=V_y^\pi$.
		
		Note that $\operatorname{supp}(V_y) \subset \overline{B_\pi}(y)$ and $V_y \in W^{2,2}(M)$, which ensures that $\mathscr{L}(i_{V_y}R^\nabla)$ is well-defined. Let $\{e_1,\dots,e_n\}$ be a local orthonormal frame for $TM$ over $B_\pi(y)$. Expressing $V_y$ in this frame:
		\begin{equation}
			\begin{split}
				V_y &= \sum_{i=1}^n V_i e_i, \\
				D_{e_i} V_y &= \sum_{j=1}^n V_{j,i} e_j, \\
				D_{e_j}D_{e_i} V_y &= \sum_{k=1}^n V_{k,ji} e_k.
			\end{split}
		\end{equation}
		The symmetry $V_{i,j} = V_{j,i}$ holds for all $i,j$, as verified by:
		\begin{align*}
			V_{i,j}= \langle D_{e_j} \operatorname{grad}(f \circ \rho_y), e_i \rangle= \operatorname{Hess}(f \circ \rho_y)(e_i,e_j)= V_{j,i}.
		\end{align*}
		Differentiating this relation with respect to $e_k$ yields $V_{i,jk} = V_{j,ik}$ for all $i,j,k$. Using the geodesic equation $D_{\operatorname{grad}(\rho_y)}\!\operatorname{grad}(\rho_y) = 0$ and the relation $V = f'(\rho_y)\operatorname{grad}(\rho_y)$, we derive
		\begin{align*}
			\sum_{i=1}^n V_i V_{j,i} = \langle D_V V, e_j \rangle = f''(\rho_y) V_j
		\end{align*}
		for any index $j$. Differentiating with respect to $e_l$ and applying the curvature commutator formula 
		$[D_{e_l}, D_{e_i}] = R(e_i, e_l) + D_{[e_i,e_l]}$ yields:
		\begin{align}
			\sum_{i=1}^n V_i V_{j,li} &= \frac{f'''(\rho_y)}{f'(\rho_y)} V_j V_l + f''(\rho_y) V_{j,l}- \sum_{i=1}^n V_{i,j} V_{i,l} - \sum_{i,k=1}^n V_i V_k R_{jkli}
		\end{align}
		where the curvature components are defined as $R_{jkli} = \langle R(e_l, e_i)e_k, e_j \rangle$. On $B_\pi(y)\backslash\{y\}$ with $f(\rho) = \sin\rho$, this simplifies to:
		\begin{align}\label{D^2V}
			\sum_{i=1}^n V_i V_{j,li} &= -V_j V_l + \cos\rho_yV_{j,l}- \sum_{i=1}^n V_{i,j} V_{i,l} - \sum_{i,k=1}^n V_i V_k R_{jkli}.
		\end{align}
		\subsection{First-order derivative estimate for the vector field $V_y$}\ 
		
		In this section, we shall estimate those terms in the second variation $\mathscr{L}(i_{V_y}R^\triangledown)$ that involve first derivatives of $V_y$. For any $x\in B_\pi(y)\backslash\{y\}$ with $\rho_y(x)=\rho$, let $\gamma(t)$ be the geodesic curve on $B_\pi(y)$ such that $\gamma(0)=y$ and $\gamma(\rho)=x$. Define projection operators $P_1,P_2:T_xM\to T_xM$ to be
		\begin{align*}
			P_1(X)=\langle X,\gamma'(\rho)\rangle\gamma'(\rho),\ P_2(X)=X-P_1(X).
		\end{align*}
		Let $\mathcal{Q}:X\mapsto D_XV_y$ be a linear map on $T_xM$. By the definition of $V_y$, we have
		\begin{align*}
			\mathcal{Q}(X)=f''(\rho)X(\rho)\operatorname{grad}(\rho_y)+f'(\rho)D_X\operatorname{grad}(\rho_y)=\cos\rho P_1(X)+\sin\rho\sum_{i=1}^n\operatorname{Hes}(\rho_y)(P_2(X),e_i)e_i.
		\end{align*}
		Suppose $\lambda_1,...,\lambda_n$ be the eigenvalues of $\mathcal{Q}$ and $e_1,...,e_n$ be the corresponding eigenvectors, where $\lambda_1=\cos\rho$ and $e_1=\operatorname{grad}(\rho_y)$. Then we have $V_{i,j}=\lambda_i\delta_{ij}$. Under the assumption that the sectional curvature satisfies $K\in(\delta,1]$ and applying the Hessian comparison theorem (Theorem A, [GH]), we derive
		\begin{equation}\label{lambda_i}
			\begin{split}
				\cos\rho\le\lambda_i&\le\sqrt\delta\sin\rho\cot(\sqrt\delta\rho)+\sin\rho(\cot\rho-\sqrt\delta\cot(\sqrt\delta\rho))P_2\\
				&\le\sqrt\delta\sin\rho\cot(\sqrt\delta\rho)
			\end{split}
		\end{equation}
		for any $i=2,...,n$, where we use $P_2$ is a non-negative operator and $t\cot(t\rho)$ is decreasing respect to $t$ for any $\rho\in(0,\pi)$. Let \(\rho_{\delta} \in \left( \frac{\pi}{2}, \frac{\pi}{2\sqrt{\delta}} \right)\) be the unique solution to \(\tan\bigl(\sqrt{\delta}\,\rho_{\delta}\bigr) + \sqrt{\delta}\, \tan \rho_{\delta} = 0\) and define
		\begin{align*}
			f_1(\rho)=\left\{
			\begin{array}{ll}
				\cos^2\rho,&\rho\in(0,\frac\pi2]\\
				\sqrt\delta\cos\rho\sin\rho\cot(\sqrt\delta\rho),&\rho\in[\frac\pi2,\frac\pi{2\sqrt\delta}]\\
				\delta\sin^2\rho\cot^2(\sqrt\delta\rho),&\rho\in[\frac\pi{2\sqrt\delta},\pi)
			\end{array}
			\right.
		\end{align*}
		\begin{align*}
			f_2(\rho)=\left\{
			\begin{array}{ll}
				\delta\sin^2\rho\cot^2(\sqrt\delta\rho),&\rho\in(0,\rho_\delta]\\
				\cos^2\rho,&\rho\in[\rho_\delta,\pi)
			\end{array}
			\right.
		\end{align*}
		\begin{align*}
			g_1(\rho)=\left\{
			\begin{array}{ll}
				\cos^2\rho,&\rho\in(0,\frac\pi2]\\
				0,&\rho\in[\frac\pi2,\frac\pi{2\sqrt\delta}]\\
				\delta\sin^2\rho\cot^2(\sqrt\delta\rho),&\rho\in[\frac\pi{2\sqrt\delta},\pi)
			\end{array}
			\right.
		\end{align*}
		\begin{align*}
			g_2(\rho)=\left\{
			\begin{array}{ll}
				\delta\sin^2\rho\cot^2(\sqrt\delta\rho),&\rho\in(0,\rho_\delta]\\
				\cos^2\rho,&\rho\in[\rho_\delta,\pi)
			\end{array}
			\right.
		\end{align*}
		\begin{align*}
			h_1(\rho)=\left\{
			\begin{array}{ll}
				\cos^2\rho,&\rho\in(0,\frac\pi2]\\
				\sqrt\delta\sin\rho\cos\rho\cot(\sqrt\delta\rho),&\rho\in[\frac\pi2,\pi)
			\end{array}
			\right.
		\end{align*}
		\begin{align*}
			h_2(\rho)=\left\{
			\begin{array}{ll}
				\sqrt\delta\sin\rho\cos\rho\cot(\sqrt\delta\rho),&\rho\in(0,\frac\pi2]\\
				\cos^2\rho,&\rho\in[\frac\pi2,\pi)
			\end{array}
			\right.
		\end{align*}
		Based on the eigenvalue estimate (\ref{lambda_i}) for $\mathcal{Q}$, we immediately obtain the following lemma:
		\begin{lemma}\label{contolled by fgh}
			For each $x \in B_\pi(y)\backslash\{y\}$ with $\rho_x(y) = \rho$, let $\lambda_1,..., \lambda_n$ denote the eigenvalues of the operator $\mathcal{Q}_x: T_xM \to \mathbb{R}$ defined by $X \mapsto D_X V_y$.
			 We have $\lambda_1=\cos\rho$ and
			\begin{equation}
				\begin{split}
					&f_1(\rho)\le\lambda_i\lambda_j\le f_2(\rho),\\
					&g_1(\rho)\le\lambda_i^2\le g_2(\rho),\\
					&h_1(\rho)\le\cos\rho\lambda_i\le h_2(\rho)
				\end{split}
			\end{equation}
			for any $i=2,...,n$.
		\end{lemma}
		
		\subsection{The second variation formula}\ 
		
		In this section, we will rewrite all terms involving second derivatives of $V_y$ in the second variation to match the form on the left-hand side of (\ref{D^2V}). We extend $V_i,V_{i,j},V_{i,jk}$ to $L^2$-functions on $M$ by setting them to vanish outside $B_\pi(y)\backslash\{y\}$.
		\begin{lemma}\label{2var-D^2V}
			Assume $R^\triangledown=\sum_{i,j=1}^nF_{ij}\omega^i\wedge\omega^j$ on $B_\pi(y)$. The last 4 terms in $\mathscr{L}(i_{V_y}R^\triangledown)$ are given by:\\
			(i)
			\begin{equation}\label{part i}
				\begin{split}
					&-\int_M\langle\nabla_{D_{e_i}V_y}R^\triangledown(e_i,e_j),R^\triangledown(V_y,e_j)\rangle dV\\
					=&\int_M\frac12|R^\triangledown|^2(V_{k,kl}V_l+V_{k,k}V_{l,l})-\frac32V_{k,k}V_{l,i}\langle F_{ij},F_{lj}\rangle+V_kV_lR_{ik}\langle F_{ij},F_{lj}\rangle+V_{k,i}V_{l,k}\langle F_{ij},F_{lj}\rangle\\
					&+\frac12V_{k,il}V_l\langle F_{ij},F_{jk}\rangle+\frac12V_sV_lR_{ksji}\langle F_{ij},F_{kl}\rangle+V_{k,i}V_{l,j}\langle F_{ij},F_{kl}\rangle dV.
				\end{split}
			\end{equation}
			(ii)
			\begin{equation}\label{part ii}
				\begin{split}
					&-\int_M\langle\nabla_{e_i}R^\triangledown(D_{e_i}V_y,e_j),R^\triangledown(V_y,e_j)\rangle dV\\
					=&\int_M\frac12|R^\triangledown|^2(V_{k,kl}V_l+V_{k,k}V_{l,l})-\frac32V_{k,k}V_{l,i}\langle F_{ij},F_{lj}\rangle+V_kV_lR_{ik}\langle F_{ij},F_{lj}\rangle+V_{k,i}V_{l,k}\langle F_{ij},F_{lj}\rangle\\
					&+\frac12V_{k,il}V_l\langle F_{ij},F_{jk}\rangle+\frac12V_sV_lR_{ksji}\langle F_{ij},F_{kl}\rangle+V_{k,i}V_{l,j}\langle F_{ij},F_{kl}\rangle dV.
				\end{split}
			\end{equation}
			(iii)
			\begin{equation}\label{part iii}
				\begin{split}
					\int_M\langle R^\triangledown(D^\ast DV_y,e_j),R^\triangledown(V_y,e_j)\rangle dV=\int_MV_{i,i}V_{l,k}\langle F_{kj},F_{lj}\rangle-\frac12|R^\triangledown|^2(V_{i,il}V_l+V_{i,i}V_{l,l})-V_iV_lR_{ik}\langle F_{kj},F_{lj}\rangle dV.
				\end{split}
			\end{equation}
			(iv)
			\begin{equation}\label{part iv}
				\begin{split}
					\int_M\langle R^\triangledown(D^2_{e_i,e_j}V_y,e_i),R^\triangledown(V_y,e_j)\rangle dV=-\frac12\int_MV_sV_lR_{jsik}\langle F_{ki},F_{jl}\rangle dV,
				\end{split}
			\end{equation}
		where $R_{ik} = \sum_{j} R_{ijik}$ be the Ricci curvature, and we omit the summation symbol for repeated indices in all expressions.
		\end{lemma}
		\begin{proof}(i)\ 
			We will perform pointwise computations on $B_\pi(y)\backslash\{y\}$. Using Bianchi identity, we have
			\begin{align*}
				&-\langle\nabla_{D_{e_i}V_y}R^\triangledown(e_i,e_j),R^\triangledown(V_y,e_j)\rangle\\
				=&-V_{k,i}V_l\langle F_{ij,k},F_{lj}\rangle\\
				=&-\operatorname{div}(g(D_{e_i}V_y,\cdot)V_l\langle F_{ij},F_{lj}\rangle)+V_{k,ik}V_l\langle F_{ij},F_{lj}\rangle+V_{k,i}V_{l,k}\langle F_{ij},F_{lj}\rangle+V_{k,i}V_l\langle F_{ij},F_{lj,k}\rangle\\
				=&-\operatorname{div}(g(D_{e_i}V_y,\cdot)V_l\langle F_{ij},F_{lj}\rangle)+V_{k,ki}V_l\langle F_{ij},F_{lj}\rangle+V_kV_lR_{ik}\langle F_{ij},F_{lj}\rangle+V_{k,i}V_{l,k}\langle F_{ij},F_{lj}\rangle\\
				&-V_{k,i}V_l\langle F_{ij},F_{jk,l}\rangle-V_{k,i}V_l\langle F_{ij},F_{kl,j}\rangle.
			\end{align*}
			Using the Yang-Mills equation $\sum_{i=1}^nF_{ij,i}=0$, we have
			\begin{equation}\label{V_kkiV_l}
				\begin{aligned}[b]
					&V_{k,ki}V_l\langle F_{ij},F_{lj}\rangle\\
					=&\operatorname{div}(V_{k,k}\langle R^\triangledown(\cdot,e_j),R^\triangledown(V_y,e_j)\rangle)-V_{k,k}V_{l,i}\langle F_{ij},F_{lj}\rangle-V_{k,k}V_l\langle F_{ij},F_{lj,i}\rangle\\
					=&\operatorname{div}(V_{k,k}\langle R^\triangledown(\cdot,e_j),R^\triangledown(V_y,e_j)\rangle)-\frac12V_{k,k}V_le_l(|R^\triangledown|^2)-V_{k,k}V_{l,i}\langle F_{ij},F_{lj}\rangle\\
					=&\operatorname{div}(V_{k,k}\langle R^\triangledown(\cdot,e_j),R^\triangledown(V_y,e_j)\rangle)-\frac12V_{k,k}g(V,\cdot)|R^\triangledown|^2)+\frac12|R^\triangledown|^2(V_{k,kl}V_l+V_{k,k}V_{l,l})-V_{k,k}V_{l,i}\langle F_{ij},F_{lj}\rangle,
				\end{aligned}
			\end{equation}
			where we use $\langle F_{ij},F_{lj,i}\rangle=-\langle F_{ij},F_{ji,l}\rangle-\langle F_{ij},F_{il,j}\rangle$ and thus $\langle F_{ij},F_{lj,i}\rangle=\frac12\langle F_{ij},F_{ij,l}\rangle=\frac12e_l(|R^\triangledown|^2)$.
			
			Since
			\begin{align*}
				&-V_{k,i}V_l\langle F_{ij},F_{jk,l}\rangle\\
				=&-\operatorname{div}(V_{k,i}g(V_y,\cdot)\langle F_{ij},F_{jk,l}\rangle)+V_{k,il}V_l\langle F_{ij},F_{jk}\rangle+V_{k,i}V_{l,l}\langle F_{ij},F_{jk}\rangle+V_{k,i}V_l\langle F_{ij,l},F_{jk}\rangle
			\end{align*}
			and note that the final term on the right-hand side is exactly the negative of the left-hand side, we have
			\begin{align*}
				-V_{k,i}V_l\langle F_{ij},F_{jk,l}\rangle=\frac12(-\operatorname{div}(V_{k,i}g(V_y,\cdot)\langle F_{ij},F_{jk,l}\rangle)+V_{k,il}V_l\langle F_{ij},F_{jk}\rangle+V_{k,i}V_{l,l}\langle F_{ij},F_{jk}\rangle).
			\end{align*}
			Applying the Yang-Mills equation and the curvature commutator formula, we have
			\begin{align*}
				&-V_{k,i}V_l\langle F_{ij},F_{kl,j}\rangle\\
				=&-\operatorname{div}(V_{k,i}\langle R^\triangledown(e_i,\cdot),R^\triangledown(e_k,V_y)\rangle)+V_{k,ij}V_l\langle F_{ij},F_{kl}\rangle+V_{k,i}V_{l,j}\langle F_{ij},F_{kl}\rangle\\
				=&-\operatorname{div}(V_{k,i}\langle R^\triangledown(e_i,\cdot),R^\triangledown(e_k,V_y)\rangle)+V_{k,ji}V_l\langle F_{ij},F_{kl}\rangle+V_sV_lR_{ksji}\langle F_{ij},F_{kl}\rangle+V_{k,i}V_{l,j}\langle F_{ij},F_{kl}\rangle\\
				=&-2\operatorname{div}(V_{k,i}\langle R^\triangledown(e_i,\cdot),R^\triangledown(e_k,V_y)\rangle)+V_sV_lR_{ksji}\langle F_{ij},F_{kl}\rangle+2V_{k,i}V_{l,j}\langle F_{ij},F_{kl}\rangle-V_{k,j}V_l\langle F_{ij},F_{kl,i}\rangle.
			\end{align*}
			Using again that the last term on the RHS precisely equals minus the LHS expression, we obtain
			\begin{align}\label{V_kiV_l}
				-V_{k,i}V_l\langle F_{ij},F_{kl,j}\rangle=\frac12V_sV_lR_{ksji}\langle F_{ij},F_{kl}\rangle+V_{k,i}V_{l,j}\langle F_{ij},F_{kl}\rangle.
			\end{align}
			We may apply the Stokes theorem to divergence-type terms.In fact, the eigenvalue estimate (\ref{lambda_i}) for $\mathcal{Q}$ implies $|V_{i,j}| \le C\rho_y^{-1}$ on $B_{\pi}(y) \backslash\{y\}$. Using Stokes theorem, we have
			\begin{align*}
				&|\int_{B_\pi(y)\backslash B_r(y)}\operatorname{div}(V_{k,k}\langle R^\triangledown(\cdot,e_j),R^\triangledown(V_y,e_j)\rangle)dV|\\
				\le&\int_{\partial B_r(y)}|V_{k,k}\langle R^\triangledown(\nu,e_j),R^\triangledown(V_y,e_j)\rangle|d\Theta\\
				\le&C\int_{B_r(y)}r^{-1}d\Theta\\
				\le&Cr^{n-2}
			\end{align*}
			for any $0<r<\pi$, where $\nu$ is the unit outer normal vector of $\partial B_r(y)$. Thus we have
			\begin{align*}
				&\int_M\operatorname{div}(V_{k,k}\langle R^\triangledown(\cdot,e_j),R^\triangledown(V_y,e_j)\rangle)dV\\
				=&\lim_{r\to0}\int_{B_\pi(y)\backslash B_r(y)}\operatorname{div}(V_{k,k}\langle R^\triangledown(\cdot,e_j),R^\triangledown(V_y,e_j)\rangle)dV\\
				=&0.
			\end{align*}
			Applying the same reasoning to the remaining divergence-form terms, we conclude that their integrals over $M$ vanish. Synthesizing the preceding computations and integration over $M$ yields
			\begin{align*}
				&-\int_M\langle\nabla_{D_{e_i}V_y}R^\triangledown(e	_i,e_j),R^\triangledown(V_y,e_j)\rangle dV\\
				=&\int_M\frac12|R^\triangledown|^2(V_{k,kl}V_l+V_{k,k}V_{l,l})-\frac32V_{k,k}V_{l,i}\langle F_{ij},F_{lj}\rangle+V_kV_lR_{ik}\langle F_{ij},F_{lj}\rangle+V_{k,i}V_{l,k}\langle F_{ij},F_{lj}\rangle\\
				&+\frac12V_{k,il}V_l\langle F_{ij},F_{jk}\rangle+\frac12V_sV_lR_{ksji}\langle F_{ij},F_{kl}\rangle+V_{k,i}V_{l,j}\langle F_{ij},F_{kl}\rangle dV.
			\end{align*}
			(ii)\ Using $V_{i,j}=V_{j,i}$, it is straightforward to verify that
			\begin{align*}
				-\nabla_{e_i}R^\triangledown(D_{e_i}V_y,e_j)=-V_{k,i}\nabla_{e_i}R^\triangledown(e_k,e_j)=-\nabla_{D_{e_k}V_y}R^\triangledown(e_k,e_j),
			\end{align*}
			thus the result in (ii) is identical to that in (i).\\
			(iii)\ Using the curvature commutator formula, we have
			\begin{align*}
				\langle R^\triangledown(D^\ast DV_y,e_j),R^\triangledown(V_y,e_j)\rangle=-V_{k,ii}V_l\langle F_{kj},F_{lj}\rangle=-V_{i,ik}V_l\langle F_{kj},F_{lj}\rangle-V_iV_lR_{ik}\langle F_{kj},F_{lj}\rangle.
			\end{align*}
			Using (\ref{V_kkiV_l}), we can process the first term and we obtain
			\begin{align*}
				\int_M\langle R^\triangledown(D^\ast DV_y,e_j),R^\triangledown(V_y,e_j)\rangle dV=\int_MV_{i,i}V_{l,k}\langle F_{kj},F_{lj}\rangle-\frac12|R^\triangledown|^2(V_{i,il}V_l+V_{i,i}V_{l,l})-V_iV_lR_{ik}\langle F_{kj},F_{lj}\rangle dV.
			\end{align*}
			(iv)\ Applying Yang-Mills equation, we have
			\begin{align*}
				&\langle R^\triangledown(D^2_{e_i,e_j}V_y,e_i),R^\triangledown(V_y,e_j)\rangle\\
				=&V_{k,ji}V_l\langle F_{ki},F_{lj}\rangle\\
				=&\operatorname{div}(V_{k,j}\langle R^\triangledown(e_k,\cdot),R^\triangledown(V_y,e_j)\rangle)-V_{k,j}V_{l,i}\langle F_{ki},F_{lj}\rangle-V_{k,j}V_l\langle F_{ki},F_{lj,i}\rangle.
			\end{align*}
			Employing (\ref{V_kiV_l}), we compute the ultimate term to obtain
			\begin{align*}
				\int_M\langle R^\triangledown(D^2_{e_i,e_j}V_y,e_i),R^\triangledown(V_y,e_j)\rangle dV=-\frac12\int_MV_sV_lR_{jsik}\langle F_{ki},F_{jl}\rangle dV.
			\end{align*}
		\end{proof}
		Lemma \ref{second variation} and lemma \ref{2var-D^2V} directly establish the following proposition.
		\begin{proposition}
			For the vector field $V_y$ defined in (\ref{V_y}), the second variation along $i_{V_y}R^\nabla$ reads
			\begin{equation}\label{2nd var locally}
				\begin{split}
					\mathscr{L}(i_{V_y}R^\triangledown)=\int_M&\frac12|R^\triangledown|^2(V_{i,il}V_l+V_{i,i}V_{,ll})+V_sV_lR_{kjis}\langle F_{ik},F_{lj}\rangle+\frac12V_sV_lR_{ksji}\langle F_{ij},F_{kl}\rangle-2V_{k,k}V_{l,i}\langle F_{ij},F_{lj}\rangle\\
					&+V_{k,il}V_l\langle F_{ij},F_{jk}\rangle+2V_{k,i}V_{l,k}\langle F_{ij},F_{lj}\rangle+2V_{k,i}V_{l,j}\langle F_{ij},F_{kl}\rangle dV.
				\end{split}
			\end{equation}
		\end{proposition}
		\section{Proof of the main result}
		In this section, we will proof the theorem \ref{main theorem}. We first need the estimate of the curvature tensors.
		\begin{proposition}[\cite{BM} Section 6]\label{estiamte of curvature tensor}
			Assume the sectional curvature $K\in(\delta,1]$, then the curvature tensor satisfies
			\begin{equation}
				\begin{split}
					&R_{ijij}\in(\delta,1],\\
					&R_{ijil}\in(-\frac12(1-\delta),\frac12(1-\delta)),\\
					&R_{ijkl}\in(-\frac23(1-\delta),\frac23(1-\delta)),
				\end{split}
			\end{equation}
			where $i,j,k,l$ are not equal to each other.
		\end{proposition}
		We have defined a family of variation directions $\{i_{V_y}R^\triangledown\mid y\in M\}$, along which we can estimate the second variations.
		\begin{proposition}\label{int L(iV_yR)}
			For any $y\in M$, let $V_y$ be the vector field defined as (\ref{V_y}). Under the assumption that $\operatorname{inj}(M)=\pi$ and the sectional curvature $K\in(\delta,1]$, we have
			\begin{align}
				\int_M\mathscr{L}(i_{V_y}R^\triangledown)dV(y)\le\int_M|R^\triangledown|^2(x)\int_{B_\pi(x)}\Phi_\delta(\rho_x(y))dV(y)dV(x),
			\end{align}
			where
			\begin{align*}
				\Phi_\delta(\rho)=&\frac12+((\frac2n+n)(1-\delta)-\frac{1-(n-1)\delta}2)\sin^2\rho+4f_2+(\frac{(n-1)^2}2+6)g_2\\
				&+\frac32(n-1)h_2-4nf_1-\frac{n-1}2g_1-2h_1.
			\end{align*}
		\end{proposition}
		\begin{proof}
			For any $x\in B_\pi(y)\backslash\{y\}$, we will estimate the second variation (\ref{2nd var locally}) term by term at $x$. Using (\ref{D^2V}) , Lemma \ref{contolled by fgh} and Proposition \ref{estiamte of curvature tensor}, we have
			\begin{align*}
				\int_M\frac12|R^\triangledown|^2V_{i,il}V_ldV&=\int_{B_\pi(y)}\frac12|R^\triangledown|^2(-\sin^2\rho+\cos\rho\sum_i\lambda_i-\sum_i\lambda_i^2-\sin^2\rho\sum_iK_{i1})dV\\
				&\le\int_{B_\pi(y)}\frac12|R^\triangledown|^2((n-1)h_2-(n-1)g_1-(1+(n-1)\delta)\sin^2\rho)dV.
			\end{align*}
			Integrating with respect to $y$ and using Fubini's theorem, we have
			\begin{equation}\label{part1}
				\begin{split}
					&\int_M\int_M\frac12|R^\triangledown|^2V_{i,il}^yV_l^ydV(y)dV(x)\\
					\le&\int_M|R^\triangledown|^2(x)\int_{B_\pi(x)}\frac12((n-1)h_2(\rho_x(y))-(n-1)g_1(\rho_x(y))-(1+(n-1)\delta)\sin^2\rho_x(y))dV(y)dV(x).
				\end{split}
			\end{equation}
			Through parallel reasoning, we have
			\begin{align*}
				V_{k,il}V_l\langle F_{ij},F_{jk}\rangle&=\sin^2\rho\sum_i|F_{1i}|^2-\cos\rho\sum_{i,j}\lambda_i|F_{ij}|^2+\sum_{i,j}\lambda_i^2|F_{ij}|^2-\sin^2\rho R_{k1i1}\langle F_{ij},F_{jk}\rangle\\
				&\le(1-\delta)\sin^2\rho\sum_i|F_{1i}|^2+ ((1+n(1-\delta))\sin^2\rho-2h_1+2g_2)|R^\triangledown|^2,
			\end{align*}
			where the final inequality follows via the estimate:
			\begin{align*}
				&-\sin^2\rho R_{k1i1}\langle F_{ij},F_{jk}\rangle\\
				\le&\sin^2\rho(\sum_{i=2}^nK_{1i}|F_{ij}|^2+\frac{1-\delta}2\sum_{i\ne k}|\langle F_{ij},F_{jk}\rangle|)\\
				\le&\sin^2\rho(2|R^\triangledown|^2-\delta|F_{ij}|^2+\frac{1-\delta}2\sum_{i,k}|\langle F_{ij},F_{kj}\rangle|-(1-\delta)|R^\triangledown|^2)\\
				\le&(1+n(1-\delta))\sin^2\rho|R^\triangledown|^2-\delta\sin^2\rho|F_{1j}|^2.
			\end{align*}
			We need to estimate the integral of $\sin^2\rho\sum_i|F_{1i}|^2$. For any $x\in M$, we identify $U_xM=\{W\in T_xM\mid |W|=1\}$ canonically with $S^{n-1}$. Specifically, let $(x^1,...,x^n)$ be the local coordinate near $x$. For any $w=(w^1,...,w^n)\in S^{n-1}\subset\mathbb{R}^n$, let $W=\sum_{i=1}^nw^i\frac\partial{\partial x^i}\in U_xM$ be the corresponding tangent vector. For any $u,w\in S^{n-1}$, define a quadratic form $\tilde Q_x(U,W)=\langle i_UR^\triangledown,i_WR^\triangledown\rangle(x)$ and denote $Q_x(u,w):=\tilde Q_x(U,W):S^{n-1}\times S^{n-1}\to\mathbb{R}$. Then we have
			\begin{align*}
				\int_{S^{n-1}}Q_x(w,w)dw=\int_{S^{n-1}}\tilde Q_{x;ij}w^iw^jdw=tr(\tilde Q_x)\int_{S^{n-1}}(w^1)^2dw=\frac{tr(\tilde Q_x)}n|S^{n-1}|=\frac2n|S^{n-1}||R^\triangledown|^2(x).
			\end{align*}
			and thus
			\begin{align*}
				&\int_M\int_{B_\pi(x)}\sin^2\rho_y(x)\sum_i|F_{1i}|^2(x)dV(y)dV(x)\\
				=&\int_M\int_0^\pi\sin^2\rho d\rho\int_{S_\rho^{n-1}}Q_x(w,w)dwdV(x)\\
				=&\frac2n\int_M|R^\triangledown|^2(x)\int_{B_\pi(x)}\sin^2\rho_x(y)dV(y).
			\end{align*}
			Hence we have
			\begin{equation}\label{part2}
				\begin{split}
					&\int_M\int_MV_{kil}V_l\langle F_{ij},F_{jk}\rangle dV(y)dV(x)\\
					\le&\int_M|R^\triangledown|^2(x)\int_{B_\pi(x)} ((1+(n+\frac2n)(1-\delta))\sin^2\rho_x-2h_1(\rho_x)+2g_2(\rho_x))dV(y)dV(x)
				\end{split}
			\end{equation}
			Applying (\ref{D^2V})  and Proposition \ref{estiamte of curvature tensor} again, we obtain
			\begin{equation}\label{part3}
				\begin{split}
					&\int_M\int_M\frac12|R^\triangledown|^2V_{i,i}V_{l,l} dV(y)dV(x)\\
					=&\int_M\int_{B_\pi(x)}\frac12|R^\triangledown|^2(\sum_i\lambda_i)^2 dV(y)dV(x)\\
					\le&\int_M|R^\triangledown|^2(x)\int_{B_\pi(x)}\frac12(\cos^2\rho_x+(n-1)^2g_2(\rho_x)+2(n-1)h_2(\rho_x)) dV(y)dV(x).
				\end{split}
			\end{equation}
			and
			\begin{equation}\label{part4}
				\begin{split}
					&\int_M\int_M-2V_{k,k}V_{i,l}\langle dV(y)dV(x) F_{ij},F_{lj}\rangle dV(y)dV(x)\\
					=&\int_M\int_{B_\pi(x)}-2\sum_k\lambda_k\sum_{i,j}\lambda_i|F_{ij}|^2 dV(y)dV(x)\\
					\le&\int_M|R^\triangledown|^2(x)\int_{B_\pi(x)}-4nf_1(\rho_x)dV(y)dV(x).
				\end{split}
			\end{equation}
			By Lemma \ref{contolled by fgh}, we have
			\begin{equation}\label{part5}
				\begin{split}
					&\int_M\int_M2V_{k,i}V_{l,k}\langle F_{ij},F_{lj}\rangle dV(y)dV(x)\\
					=&2\int_M\int_{B_\pi(x)}\sum_{i,j}\lambda_i^2|F_{ij}|^2 dV(y)dV(x)\\
					\le&\int_M|R^\triangledown|^2(x)\int_{B_\pi(x)}4g_2(\rho_x)dV(y)dV(x).
				\end{split}
			\end{equation}
			and
			\begin{equation}\label{part6}
				\begin{split}
					&\int_M\int_M2V_{k,i}V_{l,j}\langle F_{ij},F_{kl}\rangle dV(y)dV(x)\\
					=&2\int_M\int_{B_\pi(x)}\sum_{i,j}\lambda_i\lambda_j|F_{ij}|^2dV(y)dV(x)\\
					\le&\int_M|R^\triangledown|^2(x)\int_{B_\pi(x)}4f_2(\rho_x) dV(y)dV(x).
				\end{split}
			\end{equation}
			Finally, we have
			\begin{equation}\label{part7}
				\begin{split}
					&V_sV_lR_{kjis}\langle F_{ik},F_{lj}\rangle+\frac12V_sV_lR_{ksji}\langle F_{ij},F_{kl}\rangle\\
					=&\sin^2\rho(R_{kji1}-\frac12R_{1jik})\langle F_{ij},F_{1j}\rangle\\
					=&\frac12\sin^2\rho(R_{1ijk}-R_{1kij})\langle F_{ij},F_{1j}\rangle\\
					=&0.
				\end{split}
			\end{equation}
			Summing up (\ref{part1}) through (\ref{part7}), we complete the proof of the proposition.
		\end{proof}
		For each given $n$ and $\delta$, Proposition \ref{int L(iV_yR)} provides an explicit expression for $\Phi_\delta$. More precisely, if we let
		\begin{align*}
			&\Psi=\frac12+((\frac2n+n)(1-\delta)-\frac{(n-1)\delta}2)\sin^2\rho,\\
			&\phi=\delta\sin^2\rho\cot^2(\sqrt\delta\rho),\\
			&\psi=\sqrt\delta\cos\rho\sin\rho\cot(\sqrt\delta\rho),
		\end{align*}
		then we have
		\begin{align*}
			\Phi_\delta(\rho)=\left\{
			\begin{array}{ll}
				\Psi-(\frac92n+\frac32)\cos^2\rho+(\frac{(n-1)^2}2+10)\phi+\frac32(n-1)\psi,&\rho\in(0,\frac\pi2)\\
				\Psi+\frac32(n-1)\cos^2\rho+(\frac{(n-1)^2}2+10)\phi-(4n+2)\psi,&\rho\in(\frac\pi2,\rho_\delta)\\
				\Psi+(\frac{(n-1)^2}2+\frac32(n-1)+10)\cos^2\rho-(4n+2)\psi,&\rho\in(\rho_\delta,\frac\pi{2\sqrt\delta}),\\
				\Psi+(\frac{(n-1)^2}2+\frac32(n-1)+10)\cos^2\rho-(\frac92n-\frac12)\phi-2\psi,&\rho\in(\frac\pi{2\sqrt\delta},\pi).
			\end{array}
			\right.
		\end{align*}		
		To proceed with estimating the second variation, the volume comparison theorem is required.
		\begin{proposition}[\cite{IC} Theorem III.4.1, Theorem III.4.3]
			Assume $M$ is a Riemannian manifold with dimensional $n$ and sectional curvature $K\in(\delta,1]$, then we have
			\begin{align}
				\sin^{n-1}\rho dS_{n-1}d\rho\le dV\le(\frac{\sin(\sqrt\delta\rho)}{\sqrt\delta})^{n-1}dS_{n-1}d\rho,
			\end{align}
			where $dS_{n-1}$ is the volume form of standard sphere $S^{n-1}$.
		\end{proposition}
		For any $\delta\in(0,1)$, define
		\begin{align*}
			v_\delta(\rho)=\left\{
			\begin{array}{ll}
				(\frac{\sin(\sqrt\delta\rho)}{\sqrt\delta})^{n-1},&\Phi(\rho)\ge0\\
				\sin^{n-1}\rho.&\Phi(\rho)<0
			\end{array}
			\right.
		\end{align*}
		Using Proposition \ref{int L(iV_yR)}, we have
		\begin{align}\label{L(i_VR) contralled by Phi v}
			\int_M\mathscr{L}(i_{V_y}R^\triangledown)dV(y)\le|S^{n-1}|\int_0^\pi\Phi_\delta(\rho)v_\delta(\rho)d\rho\int_M|R^\triangledown|^2dV.
		\end{align}
		For any $n\ge5$, it can be verified that
		\begin{align*}
			\int_0^\pi\Phi_1(\rho)v_1(\rho)d\rho=\frac2{n^2}(4-n)<0.
		\end{align*}
		Hence 
		\begin{align}\label{delta(n)}
			\delta(n):=\inf\{\delta\in(0,1)\mid\int_0^\pi\Phi_\delta(\rho)v_\delta(\rho)d\rho<0\}
		\end{align}
		is well-defined.
		\begin{theorem}
			Let $M$ be a compact $n$-dimensional Riemannian manifold with $n\ge5$ and $\operatorname{inj}(M) \geq c_n\pi$ for some $c_n \in (0,1]$. There exists $\delta = \delta(n, \operatorname{inj}(M)) > 0$ such that if the sectional curvature satisfies $\delta < K \leq 1$, then every weakly stable Yang-Mills connection on $M$ is flat.
		\end{theorem}
		\begin{proof}
			For general injective radius $R$, using the same derivation, we can prove that
			\begin{align*}
				\int_M\mathscr{L}(i_{V^R_y}R^\triangledown)dV(y)\le|S^{n-1}|\int_0^\pi\Phi^R_\delta(\rho)v_\delta(\rho)d\rho\int_M|R^\triangledown|^2dV,
			\end{align*}
			where $\Phi_\delta^R$ is a continue map with regard to $\delta$ and $R$ with $\Phi_\delta^\pi=\Phi_\delta$. Define $c_n\in(0,1]$ to be
			\begin{align}
				c_n=\inf\{c\in(0,1]\mid \textrm{there exists }\delta\in(0,1)\textrm{ such that }\int_0^{c_n\pi}\Phi_\delta^{c_n\pi}(\rho)v_\delta(\rho)<0\}.
			\end{align}
			For any $R\ge c_n\pi$, we also define
			\begin{align}\label{delta(n,R)}
				\delta(n,R):=\inf\{\delta\in(0,1)\mid\int_0^R\Phi_\delta^R(\rho)v_\delta(\rho)d\rho<0\}.
			\end{align}
			Assume $\nabla$ be a weakly stable Yang-Mills connection. Then by (\ref{L(i_VR) contralled by Phi v}) and the definition of $\delta(n,R)$, we have
			\begin{align*}
				0\le\int_M\mathscr{L}(i_{V_y}R^\triangledown)dV(y)\le|S^{n-1}|\int_0^\pi\Phi_\delta(\rho)v_\delta(\rho)d\rho\int_M|R^\triangledown|^2dV\le0,
			\end{align*}
			which implies $R^\triangledown=0$. Then we finish the proof.
		\end{proof}
		\begin{remark}
			Assume $M$ is simply-connected. The values of the geometric constant $\delta(n)$ were computed for dimensions $n = 5$ through $20$. All computations were performed using MATLAB R2024a with the Symbolic Math Toolbox. The results are summarized in Table~1
		\end{remark}
		\begin{table}[H]  
			\centering
			\caption{Computed values of $\delta(n)$ for dimensions $5 \leq n \leq 20$}
			\label{tab:Cn_values}
			\begin{tabular}{cccc}
				\toprule
				Dimension $n$ & $\delta(n)$ & Dimension $n$ & $\delta(n)$ \\
				\midrule
				5  & 0.94888 & 13 & 0.96278 \\
				6  & 0.94551 & 14 & 0.96539 \\
				7  & 0.94582 & 15 & 0.96778 \\
				8  & 0.94789 & 16 & 0.96998 \\
				9  & 0.95073 & 17 & 0.97198 \\
				10 & 0.95384 & 18 & 0.97381 \\
				11 & 0.95697 & 19 & 0.97548 \\
				12 & 0.95997 & 20 & 0.97700 \\
				\bottomrule
			\end{tabular}
		\end{table}

\end{document}